\documentclass[letterpaper, 10 pt, conference]{ieeeconf}
\IEEEoverridecommandlockouts
\usepackage{cite}
\usepackage{amsmath,amssymb,amsfonts}
\usepackage{algorithmic}
\usepackage{graphicx}
\usepackage{textcomp}
\usepackage{xcolor}
\usepackage{amsthm}
\usepackage{subfig}
\usepackage{todonotes}

\DeclareMathOperator{\atantwo}{atan2}

\def\BibTeX{{\rm B\kern-.05em{\sc i\kern-.025em b}\kern-.08em
    T\kern-.1667em\lower.7ex\hbox{E}\kern-.125emX}}
\newtheorem{theorem}{Theorem}[]
\newtheorem{lemma}[theorem]{Lemma}
\newtheorem{proposition}[theorem]{Proposition}

\theoremstyle{definition}

\begin{document}

\title{Generalization of Optimal Geodesic Curvature Constrained Dubins' Path on Sphere with Free Terminal Orientation
\\
\thanks{$^{1}$ Deepak Prakash Kumar and Swaroop Darbha are with the Department of Mechanical Engineering, Texas A\&M University, College Station, TX 77843, USA (e-mail: {\tt\small deepakprakash1997@gmail.com, dswaroop@tamu.edu}).\\
$^{2}$ Satyanarayana Gupta Manyam is with the DCS Corporation, 4027 Col Glenn Hwy, Dayton, OH 45431, USA (e-mail: {\tt\small msngupta@gmail.com})\\
$^{3}$ David Casbeer is with the Control Science Center, Air Force Research Laboratory, Wright-Patterson Air Force Base, OH 45433 USA (e-mail:
{\tt\small david.casbeer@afrl.us.mil})\\
Distribution Statement A. Distribution is unlimited. AFRL-2024-4879 Cleared 09-05-2024.
}
}

\author{Deepak Prakash Kumar$^{1}$, Swaroop Darbha$^{1}$, Satyanarayana Gupta Manyam$^{2}$, David Casbeer$^{3}$}


\maketitle

\begin{abstract}
In this paper, motion planning for a Dubins vehicle on a unit sphere to attain a desired final location is considered. The radius of the Dubins path on the sphere is lower bounded by $r$. In a previous study, this problem was addressed, wherein it was shown that the optimal path is of type $CG, CC,$ or a degenerate path of the same for $r \leq \frac{1}{2}.$ Here, $C = L, R$ denotes an arc of a tight left or right turn of minimum turning radius $r,$ and $G$ denotes an arc of a great circle. In this study, the candidate paths for the same problem are generalized to model vehicles with a larger turning radius. In particular, it is shown that the candidate optimal paths are of type $CG, CC,$ or a degenerate path of the same for $r \leq \frac{\sqrt{3}}{2}.$
Noting that at most two $LG$ paths and two $RG$ paths can exist for a given final location, this article further reduces the candidate optimal paths by showing that only one $LG$ and one $RG$ path can be optimal, yielding a total of seven candidate paths for $r \leq \frac{\sqrt{3}}{2}.$ Additional conditions for the optimality of $CC$ paths are also derived in this study. 
\end{abstract}

\section{Introduction}

With the increasing utilization of fixed-wing Unmanned Aerial Vehicles (UAVs) for civilian and military applications, path planning for such vehicles to travel from one location and orientation (denoted together as a configuration) to another with minimum time or fuel is critical.
Since such vehicles cannot hover mid-air, and have a bound on the rate of change of their heading, they are typically modeled as a Dubins vehicle \cite{Dubins}. 
A Dubins vehicle is a vehicle that travels at a unit speed and has a bound on the curvature of the path. For such a vehicle, Dubins \cite{Dubins} showed that the minimum time/length path to travel from one configuration to another on a plane is of type $CSC,$ $CCC,$ or a degenerate path of the same. 
Here, $C = L, R$ denotes an arc of a left or a right turn of minimum turning radius, and $S$ denotes a straight-line segment. A different model for vehicles with similar curvature constraints and that can move forwards or backward, known as the Reeds-Shepp vehicle, has also been explored in the literature \cite{Reeds_Shepp}.

The results of \cite{Dubins} and \cite{Reeds_Shepp} were derived without resorting to Pontryagin's Maximum Principle (PMP) \cite{PMP}, a first-order necessary condition for optimality. Studies such as \cite{sussman_geometric_examples} and \cite{boissonat} later studied the same problems using PMP and derived the results systematically. In addition, using phase portraits with PMP in \cite{phase_portrait_kaya, weighted_problem, monroy} yields a simpler method to obtain the optimal paths by breaking down the analysis into a few cases.

While many variants of path planning on a plane have been studied in the literature, fewer studies have explored such problems in 3D. For instance, in \cite{sussman_3D}, a total curvature-constrained 3D Dubins problem is considered in which the vehicle must travel from a given location and heading direction to another. However, the complete configuration of the vehicle in 3D is not considered in this study. Recently, more studies have explored special cases of motion planning in 3D. For instance, \cite{monroy} studies motion planning on a Riemannian manifold for a Dubins vehicle. In particular, the author showed that the Dubins result on a plane generalizes to a unit sphere for $r = \frac{1}{\sqrt{2}},$ where $r$ is the radius of the tightest turn on the sphere. Studies such as \cite{time_optimal_synthesis_SO3} and \cite{time_optimal_control_satellite} study a generic time-optimal control problem on $SO(3)$ with one and two control inputs, respectively; however, the bound on the control input is assumed to be equal to one, which corresponds to $r = \frac{1}{\sqrt{2}}$. In \cite{3D_Dubins_sphere}, the authors modeled the vehicle's configuration using a Sabban frame and proved that the optimal path is of type $CGC, CCC,$ or a degenerate path for $r \leq \frac{1}{2}.$ In our recent work, a free-terminal orientation variant of the problem in \cite{3D_Dubins_sphere} was studied in \cite{free_terminal_sphere}, wherein only the final location is specified. We showed that the optimal path is of type $CG, CC,$ or a degenerate path of the same for $r \leq \frac{1}{2}$. Furthermore, the arc angle of the second $C$ segment was shown to be greater than or equal to $\pi$ for the $CC$ paths.

However, from \cite{free_terminal_sphere}, it can be observed that the candidate optimal paths for the instances with larger turn radius are not known. Though a sufficient list of candidate paths was obtained in \cite{free_terminal_sphere}, it is desired to obtain a smaller list of paths, if possible. Doing so simplifies tackling variants of the free terminal problem, such as when the sphere is rotating. 
Additionally, in \cite{free_terminal_sphere}, the case of abnormal controls is not completely characterized.
To this end, the main contributions of this article are as follows:
\begin{enumerate}
    \item Completing the argument for abnormal controls in \cite{free_terminal_sphere} by posing the optimal control problem on a Lie group.
    \item Demonstrating the construction of at most two $CG$ paths for a given final location and proving non-optimality of the second $CG$ path for all $r \in (0, 1)$.
    \item Proving non-optimality of a $C_\pi C_{\pi + \beta}$ path for $r \leq \frac{\sqrt{3}}{2}$.
    \item Proving non-optimality of a $C_\alpha C_\pi$ path for $r \leq \frac{\sqrt{3}}{2}$ for $\alpha \approx 0,$ thereby showing that the arc angle of the second $C$ segment is greater than $\pi$ in an optimal $CC$ path. 
\end{enumerate}
Hence, from this article, the candidate optimal paths are shown to be of type $CG, CC,$ or a degenerate path of the same for $r \leq \frac{\sqrt{3}}{2}.$


\section{Model for Dubins vehicle on a Unit Sphere and Candidate Optimal Paths}

In \cite{3D_Dubins_sphere}, a Dubins vehicle on a unit sphere was modeled using a Sabban frame, wherein the vehicle's location/trajectory $\mathbf{X} (s)$ was parameterized in terms of the arc length $s$. Further, $\mathbf{T} (s)$ denoted the unit tangent vector that represents the vehicle's heading. Since $\mathbf{X} (s) \cdot \mathbf{T} (s) = 0,$ $\mathbf{N} (s) := \mathbf{X} (s) \times \mathbf{T} (s)$ was defined to form an orthonormal basis. 
Hence, the rotation matrix $\mathbf{R} (s) = (\mathbf{X} (s) \quad \mathbf{T} (s) \quad \mathbf{N} (s))$ was constructed and used to represent the vehicle's configuration. The evolution of the three vectors using the Sabban frame were given by
\begin{align*}
    \frac{d \mathbf{X} (s)}{ds} &= \mathbf{T} (s), \quad \frac{d \mathbf{T} (s)}{ds} = -\mathbf{X} (s) + u_g (s) \mathbf{N} (s), \\
    \frac{d \mathbf{N} (s)}{ds} &= -u_g (s) \mathbf{T} (s),
\end{align*}
where $u_g \in [-U_{max}, U_{max}],$ the geodesic curvature, was considered to be the control input. Here, $U_{max}$ was considered to be a parameter of the vehicle. 

Since the rotation matrices form a Lie group $SO (3)$, the optimal control problem was posed on this group in our recent work \cite{kumar2024equivalencedubinspathsphere}, wherein the differential equations were rewritten as
\begin{align*}
    \frac{dg}{ds} = \overrightarrow{l}_1 \left(g (s) \right) - u_g (s) \overrightarrow{L}_{12} \left(g (s) \right),
\end{align*}
where $g (s) = \mathbf{R} (s),$ and $\overrightarrow{l}_1 (g (s)) = g (s) l_1$ and $\overrightarrow{L}_{12} (g (s)) = g (s) L_{12}$ are left-invariant vector fields, whose value at the identity of the group (the Lie algebra) are \cite{kumar2024equivalencedubinspathsphere}
\begin{align*}
    l_1 = \begin{pmatrix}
        0 & -1 & 0 \\
        1 & 0 & 0 \\
        0 & 0 & 0
    \end{pmatrix}, \quad L_{12} = \begin{pmatrix}
        0 & 0 & 0 \\
        0 & 0 & 1 \\
        0 & -1 & 0
    \end{pmatrix}.
\end{align*}
The Hamiltonian was constructed as \cite{kumar2024equivalencedubinspathsphere}
\begin{align} \label{eq: Hamiltonian}
    H = -\lambda + h_1 \left(\zeta (s) \right) - \kappa (s) H_{12} \left(\zeta (s) \right),
\end{align}
where $h_1 := \langle\zeta (s), \overrightarrow{l}_1 (g (s)) \rangle$ and $H_{12} := \langle\zeta (s), \overrightarrow{L}_{12} (g (s))\rangle$ are Hamiltonian functions
and $\kappa$ is the optimal geodesic curvature. Here, $\zeta$ at every $s$ is a dual-vector/covector. Furthermore, using PMP \cite{monroy}, $\lambda$ is a constant, and $\lambda = 0, 1.$ Additionally, from \cite{monroy, kumar2024equivalencedubinspathsphere}, the expressions for the derivative of $h_1, h_2,$ and $H_{12}$ with respect to $s$ are given by
\begin{align} \label{eq: evolution_hamiltonian_functions}
\begin{split}
    \dot{h}_1 (\zeta (s)) &= -\kappa (s) h_2 (\zeta (s)), \,\, \dot{H}_{12} (\zeta (s)) = -h_2 (\zeta (s)), \\
    \dot{h}_2 (\zeta (s)) &= H_{12} (\zeta (s)) + \kappa (s) h_1 (\zeta (s)).
\end{split}
\end{align}
From PMP, the optimal control $\kappa$ is such that $H$ is pointwise maximized. Hence, $H_{12} (s)$ dictates the optimal control $\kappa (s)$ \cite{kumar2024equivalencedubinspathsphere}. The optimal control action for $H_{12} \neq 0$ can be immediately obtained using \eqref{eq: Hamiltonian} as $\kappa (s) = -U_{max} \text{sgn} (H_{12} (s)),$
where sgn denotes the signum function. It is now desired to determine the optimal control action for $H_{12} = 0$. To this end, the following lemma from \cite{kumar2024equivalencedubinspathsphere} is utilized.
\begin{lemma}
    If $H_{12} \equiv 0,$ then $\lambda$ cannot be zero; further, for $\lambda = 1, \kappa \equiv 0.$ (Lemma~1 in \cite{kumar2024equivalencedubinspathsphere}).
\end{lemma}
Hence, the optimal control actions are as follows:
\begin{align} \label{eq: optimal_controls}
    \kappa (s) \equiv \begin{cases}
        - U_{max} \text{sgn} (H_{12} (s)), & H_{12} (s) \neq 0, \lambda = 0, 1 \\
        0, & H_{12} (s) \equiv 0, \lambda = 1.
    \end{cases}
\end{align}
It should be noted that $\kappa = -U_{max}, U_{max}$ correspond to a tight right and left turn, respectively, of radius $r = \frac{1}{\sqrt{1 + U_{max}^2}},$ and $\kappa = 0$ corresponds to a great circular arc \cite{3D_Dubins_sphere}. Hence, the optimal path is a concatenation of such segments.

For the considered free terminal problem, it is claimed that $H_{12} (L) = 0,$ where $L$ denotes the path length. Such a condition will be crucial to derive a sufficient list of candidate paths.

\begin{lemma} \label{lemma: transversality_condition}
    For the free terminal orientation problem, $H_{12} (L) = 0,$ where $L$ is the length of the path.
\end{lemma}
\begin{proof}
    The transversality condition from PMP dictates that
    \begin{align} \label{eq: transversality_condition}
        \langle \zeta (L), v \rangle = 0, \quad \forall v \in T_{g (L)} S,
    \end{align}
    where $S$ represents the submanifold in which the final configuration lies, and $T_{g (L)} S$ is its tangent plane. For the considered problem, since the final location $\mathbf{X}_f$ is specified, the submanifold is given by
    \begin{align*}
        S &= \{\begin{pmatrix}
            \mathbf{X}_f \,\, \cos{\theta} \mathbf{T}_f - \sin{\theta} \mathbf{N}_f \,\, \sin{\theta} \mathbf{T}_f + \cos{\theta} \mathbf{N}_f 
        \end{pmatrix} \},
    \end{align*}
    where $\theta \in [0, 2 \pi)$. In the above equation, $\mathbf{T}_f$ and $\mathbf{N}_f$ are tangent and tangent-normal vectors arbitrarily picked at $\mathbf{X}_f.$ It is claimed that this submanifold is generated by the vector field $\overrightarrow{L}_{12}.$ Noting that $\overrightarrow{L}_{12}$ represents rotations about $\mathbf{X},$ the rotation matrices obtained from $\overrightarrow{L}_{12}$ are given by
    \begin{align*}
        &\mathbf{R} (s) = g (s) e^{L_{12} s} = g (s) \begin{pmatrix}
            1 & 0 & 0 \\
            0 & \cos{s} & \sin{s} \\
            0 & -\sin{s} & \cos{s}
        \end{pmatrix} \\
        &= \begin{pmatrix}
            \mathbf{X} (s) \,\, \cos{s} \mathbf{T} (s) - \sin{s} \mathbf{N} (s) \,\, \sin{s} \mathbf{T} (s) + \cos{s} \mathbf{N} (s)
        \end{pmatrix}.
    \end{align*}
    Hence, at $\mathbf{X} (s) = \mathbf{X}_f,$ the vector field $\overrightarrow{L}_{12}$ generates the desired submanifold. However, since $\overrightarrow{L}_{12}$ is a vector field, 
    $\overrightarrow{L}_{12} (g (L))$ is tangential to $S.$ Hence, \eqref{eq: transversality_condition} reduces to
    \begin{align*}
        \langle \zeta (L), \overrightarrow{L}_{12} (L) \rangle := H_{12} (L) = 0,
    \end{align*}
    by the definition of $H_{12}$ \cite{kumar2024equivalencedubinspathsphere, monroy}. 
\end{proof}
Now, it is desired to perform a systematic analysis to derive the candidate optimal paths for the considered free terminal problem. In this regard, the following theorem is desired to be shown:
\begin{theorem} \label{theorem: main_theorem_free_terminal}
    A total of seven candidate optimal paths for the free terminal orientation problem on a unit sphere exist for $r \leq \frac{\sqrt{3}}{2},$ and are of type $CG, CC,$ or a degenerate path of the same. Furthermore, for the $CG$ paths, the arc angle of the $G$ segment is less than or equal to $\pi,$ and for the $CC$ paths, the arc angle of the second $C$ segment is greater than $\pi.$
\end{theorem}
To derive the candidate optimal paths, two cases can considered depending on the value of $\lambda:$ $\lambda = 0,$ and $\lambda = 1.$

\subsection{Abnormal controls ($\lambda = 0$)}

Since the optimal control action depends on $H_{12},$ it is desired to determine its evolution to obtain the candidate optimal path types. Using the expression for $\dot{H}_{12}$ given in \eqref{eq: evolution_hamiltonian_functions}, and noting that it is differentiable, the expression for $\Ddot{H}_{12}$ can be derived to be \cite{kumar2024equivalencedubinspathsphere}
\begin{align} \label{eq: diff_equation_H12}
    \frac{d^2 H_{12} (s)}{ds^2} + \left(1 + \kappa^2 (s) \right) H_{12} (s) = 0. 
\end{align}
where $\lambda = 0$ and the expression for the Hamiltonian in \eqref{eq: Hamiltonian} were used. Since $\kappa$ is piecewise constant from \eqref{eq: optimal_controls}, the solution for $H_{12}$ and $\dot{H}_{12}$ can be obtained for each interval wherein $\kappa$ is constant. 
In this case, the following propositions and lemma can be shown. 

\begin{proposition} \label{prop: optimal_path_e_0}
    For $\lambda = 0,$ the optimal path is a concatenation of $C$ segments.
\end{proposition}
\begin{proof}
    The proof follows using the same argument used to establish Proposition~3 in \cite{free_terminal_sphere}.
\end{proof}

\begin{lemma} \label{lemma: angle_C_segment_e_0}
    For $\lambda = 0,$ the arc angle of a $C$ segment that is completely traversed is exactly $\pi$ radians. 
\end{lemma}
\begin{proof}
    The proof follows from Lemma~4's proof in \cite{free_terminal_sphere}.
\end{proof}

\begin{proposition} \label{prop: candidate_paths_abnormal_controls}
    The candidate optimal path for $\lambda = 0$ is of type $C_\alpha C_\pi C_\pi \cdots C_\pi.$
\end{proposition}
\begin{proof}
    The proof follows by first noting that all intermediary $C$ segments have an arc angle of $\pi$ radians using Lemma~\ref{lemma: angle_C_segment_e_0}. Moreover, the final $C$ segment must also be completely traversed since $H_{12}$ is zero at inflection points and $H_{12} (L) = 0,$ similar to the argument used in \cite{free_terminal_sphere}. Using these two arguments, the proof is completed.
\end{proof}

It is now desired to derive a sufficient list of paths for $\lambda = 0$. To this end, it is claimed that the optimal path is of type $C_\alpha$ for $r \leq \frac{\sqrt{3}}{2},$ where $\alpha \leq \pi.$
\begin{lemma} \label{lemma: non-optimality_C_alpha_C_pi_path}
    $C_\alpha C_\pi$ path is non-optimal for $r \leq \frac{\sqrt{3}}{2}, \alpha \leq \pi$.
\end{lemma}
The proof will be provided in Section~\ref{subsect: abnormal_controls_CCpi_non_optimality} since it requires an additional result pertaining to the existence of a $CG$ path, which is derived in Section~\ref{sect: second_CG_path}.

\subsection{Normal controls ($\lambda = 1$)}

The candidate optimal path types for normal controls can be derived using the following lemmas.


\begin{lemma}
    For $\lambda = 1,$ a path of type $GC$ is non-optimal.
\end{lemma}
\begin{proof}
    The proof follows from the proof of Lemma~8 in \cite{free_terminal_sphere}.
\end{proof}
Hence, using the above lemma, it follows that if a $G$ segment is part of the path, then the optimal path can be only $G$ or $CG.$ Now, it is desired to characterize the optimal path that is made up of a concatenation of $C$ segments.
\begin{lemma} \label{lemma: angle_greater_than_pi_CCC}
     If inflection occurs at $s=s_i$ on the optimal path, then $h_1 (s_i) = 1$ and $H_{12} (s_i) = 0$. Furthermore, if $s_1, s_2$ are consecutive inflection points corresponding to a $CCC$ path, then $s_2-s_1 > \pi r$, i.e., the middle segment must have a length greater than $\pi r$.
\end{lemma}
\begin{proof}
    The proof follows from the proof for Lemma~3.3 in \cite{3D_Dubins_sphere}.
\end{proof}

\begin{lemma}
    For $\lambda = 1,$ a concatenation of $C$ segments can be a candidate optimal path if it is of type $C_\alpha C_{\pi + \beta} C_{\pi + \beta} \cdots C_{\pi + \beta},$ where $\alpha, \beta > 0$.   
\end{lemma}
\begin{proof}
    The proof follows using Lemma~\ref{lemma: angle_greater_than_pi_CCC} and the proof of Lemma~10 in \cite{free_terminal_sphere} to show that the intermediary $C$ segments and the final $C$ segment, respectively, have an arc angle of $\pi + \beta$ radians.
\end{proof}

It is now claimed that a subpath of a $C_\alpha C_{\pi + \beta} C_{\pi + \beta}$ path is non-optimal for $r \leq \frac{\sqrt{3}}{2}.$ To show this claim, regions of optimality of a $C_\alpha C_{\pi + \beta}$ path will first be shown, which is inspired by a similar argument for the planar problem in \cite{bui_free_terminal}.
\begin{lemma} \label{lemma: regions_optimality_CC_pi_beta}
    An $L_\alpha R_{\pi + \beta}$ ($R_\alpha L_{\pi + \beta}$) path can be optimal only if the final location lies inside the $R$ ($L$) segment.
\end{lemma}
\begin{proof}
    Consider $\alpha \in (0, 2 \pi)$ and $\beta \in (0, \pi)$ for an $L_\alpha R_{\pi + \beta}$ path such that $\mathbf{X}_f$ lies outside the $R$ segment. For such a final location, an alternate shorter $LR$ path can be constructed wherein the arc angle of the $L$ segment
    is less than $\alpha,$ as shown in cyan in Fig.~\ref{fig: depiction_non_optimality_LR}.
    Using spherical convexity, it follows that the alternate path is of a lower length than the original $L_\alpha R_{\pi + \beta}$ path (shown in orange).
    Hence, the $L_\alpha R_{\pi + \beta}$ path is non-optimal.
    A similar argument can be used for an $R_\alpha L_{\pi + \beta}$ path.
    \begin{figure}[htb!]
        \centering
        \includegraphics[width=0.4\linewidth]{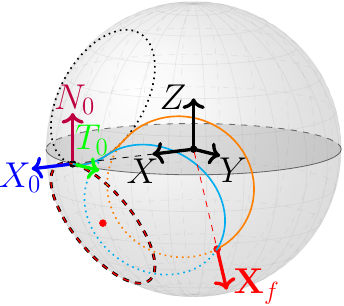}
        \caption{Depiction of non-optimality of $L_\alpha R_{\pi + \beta}$ path when $\mathbf{X}_f$ lies outside the $R$ segment using a shorter $LR$ path}
        \label{fig: depiction_non_optimality_LR}
    \end{figure}
\end{proof}

It is now desired to show the non-optimality of a subpath of a $C_\alpha C_{\pi + \beta} C_{\pi + \beta}$ path for $r \leq \frac{\sqrt{3}}{2}$.

\begin{lemma} \label{lemma: non_optimality_Cpi_Cpi_beta}
    A $C_\pi C_{\pi + \beta}$ path is non-optimal for $r \leq \frac{\sqrt{3}}{2}.$
\end{lemma}
\begin{proof}
    Consider a $L_\pi R_{\pi + \beta}$ path. For $r < \frac{\sqrt{3}}{2},$ $\mathbf{X}_f$ always lies outside the $R$ segment corresponding to the initial configuration, as can be observed from the side view of the segments shown in Fig.~\ref{subfig: LR_path_non_optimality_r_0_5}. For $r = \frac{\sqrt{3}}{2},$ $\mathbf{X}_f$ lies on the $R$ segment corresponding to the initial configuration for $\beta \rightarrow 0^-$, as shown in Fig.~\ref{subfig: LR_path_non_optimality_r_sqrt3_2}. For a marginally larger $r$, there exists a $\beta$ for which $\mathbf{X}_f$ will lie inside the $R$ segment. Hence, for $r \leq \frac{\sqrt{3}}{2},$ the $L_\pi R_{\pi + \beta}$ path is non-optimal using Lemma~\ref{lemma: regions_optimality_CC_pi_beta}. A similar argument applies for a $R_\pi L_{\pi + \beta}$ path.
    \begin{figure}[htb!]
        \centering
        \subfloat[$r < \frac{\sqrt{3}}{2}$]{\includegraphics[width = 0.4\linewidth]{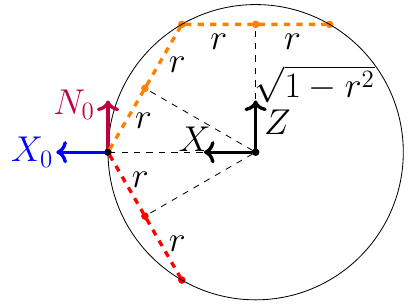}
        \label{subfig: LR_path_non_optimality_r_0_5}} \hfil
        \subfloat[$r = \frac{\sqrt{3}}{2}$]{\includegraphics[width = 0.4\linewidth]{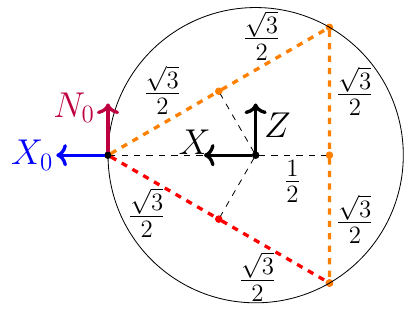}
        \label{subfig: LR_path_non_optimality_r_sqrt3_2}}
        \caption{Side view of segments of the $L_\pi R_{\pi + \beta}$ path (in orange) and $R$ segment at the initial configuration (in red)}
        \label{fig: LR_path_non_optimality_depiction}
    \end{figure}
\end{proof}

In the following section, the existence of two $LG$ (and $RG$) paths to attain a desired final location and the non-optimality of the second $LG$ ($RG$) path will be shown to reduce the number of candidate optimal paths.

\section{Non-optimality of second $CG$ path} \label{sect: second_CG_path}

In this section, two feasible $LG$ paths connecting a given initial configuration, which is chosen to be the identity matrix without loss of generality, and a final location $\mathbf{X}_f$ with coordinates $(\beta_1, \beta_2, \beta_3)$ will first be derived. The final locations (or regions on the sphere) for which an $LG$ path does not exist will then be shown in Lemma~\ref{lemma: regions_non_existence_LG_path}, using which the non-optimality of the second $LG$ path will be shown in Lemma~\ref{lemma: non-optimality_second_LG_path}. It should be noted that the analysis for an $RG$ path follows similarly, and it is ignored for brevity.

Consider an $LG$ path as shown in Fig.~\ref{fig: depiction_LG_path}, wherein the arc angle of the $L$ and $G$ segments are $\phi_1$ and $\phi_2$, respectively. For a given final location, it is desired to obtain the solution(s) for $\phi_1$ and $\phi_2$. Since the $G$ segment corresponds to a rotation about $\mathbf{N},$ the equation to be satisfied is given by
\begin{align} \label{eq: equation_LG_segment}
    \mathbf{N}_1 \cdot \mathbf{X}_f = 0,
\end{align}
where $\mathbf{N}_1 = \mathbf{N} (\phi_1)$ represents the tangent-normal after traveling an arc angle $\phi_1$ along the $L$ segment. 
\begin{figure}[htb!]
    \centering
    \includegraphics[width = 0.4\linewidth]{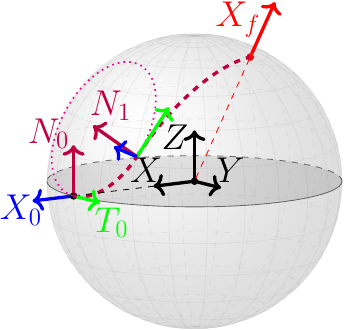}
    \caption{Solving for arc angles of the $LG$ path}
    \label{fig: depiction_LG_path}
\end{figure}

Let $\mathbf{R}_L$ denote the rotation matrix corresponding to the $L$ segment. Hence, $\mathbf{N} (\phi_1) = \mathbf{R}_L (\phi_1) e_3.$ Here, $e_i$ denotes the $i\textsuperscript{th}$ column vector of the identity matrix. Computing $\mathbf{N} (\phi_1)$ using the expression for $\mathbf{R}_L$ given in \cite{free_terminal_sphere}.
and substituting the obtained expression for $\mathbf{N} (\phi_1)$ in \eqref{eq: equation_LG_segment},
the equation obtained is given by
\begin{align} \label{eq: equation_phi_1_LG_path}
\begin{split}
    &c \phi_1 \left(\left(1 - r^2 \right) \beta_{3} - r \sqrt{1 - r^2} \beta_{1} \right) - s \phi_1 \sqrt{1 - r^2} \beta_{2} \\
    &= -r^2 \beta_{3} - r \sqrt{1 - r^2} \beta_{1}.
\end{split}
\end{align}
Here, $c \phi_1 := \cos{\phi_1}, s \phi_1 := \sin{\phi_1}.$ Suppose a solution for $\phi_1$ can be obtained from \eqref{eq: equation_phi_1_LG_path}. Then, using the rotation matrix corresponding to the $G$ segment ($\mathbf{R}_G$), which is derived in \cite{free_terminal_sphere}, the equation obtained for $\phi_2$ is given by
\begin{align} \label{eq: solving_phi2_LG_path}
    \mathbf{X}_f = \mathbf{R}_L (\phi_1) \mathbf{R}_G (\phi_2) e_1
    = c \phi_2 \mathbf{X} (\phi_1) + s \phi_2 \mathbf{T} (\phi_1),
\end{align}
Using \eqref{eq: equation_LG_segment}, it follows that $\mathbf{X} (\phi_1), \mathbf{T} (\phi_1), \mathbf{X}_f$ lie in the same plane.
Hence, a unique solution for $\phi_2$ can be obtained for a given solution for $\phi_1$ from \eqref{eq: solving_phi2_LG_path} since $\mathbf{X} (\phi_1) \cdot \mathbf{T} (\phi_1) = 0$.

Therefore, an $LG$ path does not exist if a solution for $\phi_1$ cannot be obtained from \eqref{eq: equation_phi_1_LG_path}. The final locations for which an $LG$ path does not exist 
are stated in the following lemma.

\begin{lemma} \label{lemma: regions_non_existence_LG_path}
    An $LG$ path does not exist if the final location or its antipodal point lies inside the circle corresponding to the $L$ segment.
\end{lemma}
\begin{proof}
    Consider solving for $\phi_1$ using \eqref{eq: equation_phi_1_LG_path}. A solution for $\phi_1$ cannot be obtained if
    \begin{itemize}
        \item The coefficient of $c \phi_1$ and $s \phi_1$ are both zero, and the right-hand side (RHS) of \eqref{eq: equation_phi_1_LG_path} is non-zero, or
        \item At least one of the coefficients of $c \phi_1$ and $s \phi_1$ is non-zero, and the sum of squares of these two coefficients is less than the square of the RHS.
    \end{itemize}
    In the first case, $\beta_{1}, \beta_{2}, \beta_{3}$ such that both coefficients are zero from \eqref{eq: equation_phi_1_LG_path} is 
    $(\beta_1, \beta_2, \beta_3) = \pm \left(\sqrt{1 - r^2}, 0, r \right)$. Here, $\sum_{i = 1}^3 \beta_i^2 = 1$ was used since a unit sphere is considered. However, the RHS corresponding to these solutions in \eqref{eq: equation_phi_1_LG_path} is non-zero.
    Hence, no solution for $\phi_1$ can be obtained.

    A solution for $\phi_1$ cannot be obtained in the second case from \eqref{eq: equation_phi_1_LG_path} if $\left(\left(1 - r^2 \right) \beta_{3} - r \sqrt{1 - r^2} \beta_{1} \right)^2 + \left(1 - r^2 \right) \beta_2^2 < \left(-r^2 \beta_3 - r \sqrt{1 - r^2} \beta_1 \right)^2.$
    An equivalent condition can be obtained by expanding both sides and simplifying as
    \begin{align} \label{eq: condition_nonexistence_phi1_LG_path}
        \left|\sqrt{1 - r^2} \beta_1 + r \beta_3 \right| > \sqrt{1 - r^2}.
    \end{align}
    However, it is claimed that such final locations lie inside the circle corresponding to the $L$ segment or in its diametrically opposite circle, as shown in Fig.~\ref{subfig: fin_loc_LG_path_non_existence}. This is because the plane containing the $L$ segment and the diametrically opposite circle can be derived to be $\sqrt{1 - r^2} \beta_1 + r \beta_3 = \sqrt{1 - r^2}$ and  $\sqrt{1 - r^2} \beta_1 + r \beta_3 = -\sqrt{1 - r^2},$ respectively
    (refer to Fig.~\ref{subfig: fin_loc_LG_path_side_view}). 
    Using the derived equations of the planes and Fig.~\ref{fig: final_local_LG_path_boundary}, it follows that final locations that satisfy \eqref{eq: condition_nonexistence_phi1_LG_path} lie inside the $L$ segment or in the diametrically opposite circle.
    \begin{figure}[htb!]
        \centering
        \subfloat[Boundaries for which $LG$ path does not exist for $r = 0.5$]{\includegraphics[width = 0.4\linewidth]{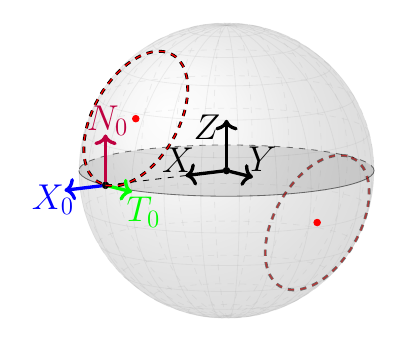}
        \label{subfig: fin_loc_LG_path_non_existence}} \hfill
        \subfloat[Side view of $L$ segment and its diametrically opposite circle]{\includegraphics[width = 0.4\linewidth]{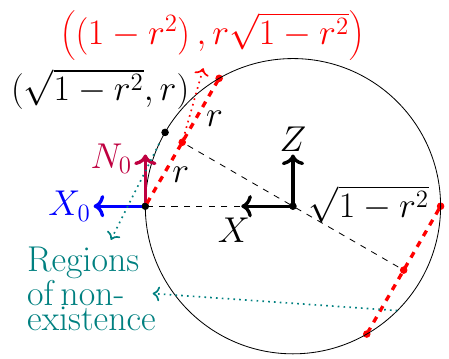}
        \label{subfig: fin_loc_LG_path_side_view}}
        \caption{Boundaries for non-existence of $LG$ path for $r = 0.5$}
        \label{fig: final_local_LG_path_boundary}
    \end{figure}

    Noting that $(\beta_1, \beta_2, \beta_3) = \pm (\sqrt{1 - r^2}, 0, r)$ also satisfies the inequality in \eqref{eq: condition_nonexistence_phi1_LG_path}, the claim in the lemma is proved.
\end{proof}

From Lemma~\ref{lemma: regions_non_existence_LG_path}, it follows that at least one solution exists for $\phi_1$ when the final location is such that \eqref{eq: condition_nonexistence_phi1_LG_path} is not satisfied. For such final locations, \eqref{eq: equation_phi_1_LG_path} can be rewritten by dividing both sides by $D = \big(\left(\left(1 - r^2 \right) \beta_3 - r \sqrt{1 - r^2} \beta_1 \right)^2 + \left(1 - r^2 \right) \beta_2^2\big)^{\frac{1}{2}},$ and defining\footnotemark
\footnotetext{If $D = 0$, then the $LG$ path does not exist for the corresponding final location, as shown in Lemma~\ref{lemma: regions_non_existence_LG_path}. Hence, $D \neq 0$ when the $LG$ path exists.}
\begin{align} \label{eq: definition_gamma_LG_path}
\begin{split}
    c \gamma := \frac{\left(1 - r^2 \right) \beta_3 - r \sqrt{1 - r^2} \beta_1}{D}, \,\,
    s \gamma := \frac{\sqrt{1 - r^2} \beta_2}{D},
\end{split}
\end{align}
as
\begin{align} \label{eq: expression_solution_phi_1_LG_path}
    \cos{\left(\phi_1 + \gamma \right)} = \frac{-r^2 \beta_3 - r \sqrt{1 - r^2} \beta_1}{D} := \delta.
\end{align}
At most two solutions can be obtained for $\phi_1$ from \eqref{eq: expression_solution_phi_1_LG_path} as
\begin{align}
    \phi_1^{s1} &= \bmod \left(\cos^{-1} \left(\delta \right) - \gamma, 2 \pi\right), \label{eq: solution_1_phi1_LG_path} \\
    \phi_1^{s2} &= \bmod \left(2 \pi - \cos^{-1} \left(\delta \right) - \gamma, 2 \pi\right). \label{eq: solution_2_phi1_LG_path}
\end{align}
Corresponding to each solution for $\phi_1$, a unique solution for $\phi_2$ can be obtained from \eqref{eq: solving_phi2_LG_path} since $\cos{\phi_2} = \mathbf{X}_f \cdot \mathbf{X} (\phi_1),$ $\sin{\phi_2} = \mathbf{X}_f \cdot \mathbf{T} (\phi_1).$
Using the expressions for $\mathbf{X} (\phi_1)$ and $\mathbf{T} (\phi_1)$ from $\mathbf{R}_L$ given in \cite{free_terminal_sphere}, $\cos{\phi_2}$ and $\sin{\phi_2}$ can be obtained as
\begin{align}
\begin{split} \label{eq: expression_cos_phi2_LG_path}
    \cos{\phi_2} &= \left(1 - \left(1 - \cos{\phi_1} \right) r^2 \right) \beta_1 + r \sin{\phi_1} \beta_2 \\
    & \quad\, + \left(1 - \cos{\phi_1} \right) r \sqrt{1 - r^2} \beta_3,
\end{split} \\
    \sin{\phi_2} &= - r \sin{\phi_1} \beta_1 + \cos{\phi_1} \beta_2 + \sin{\phi_1} \sqrt{1 - r^2} \beta_3. \label{eq: expression_sin_phi2_LG_path}
\end{align}
Noting that $\phi_2 \in [0, 2 \pi),$ its expression can be obtained as
\begin{align} \label{eq: solution_phi2_LG_path}
    \phi_2 = \bmod \left(\atantwo \left(\sin{(\phi_2)}, \cos{(\phi_2)} \right), 2 \pi \right),
\end{align}
where $\bmod$ denotes the modulus function, and $\atantwo(y, x): \mathbb{R}^2 - \{(0, 0) \} \rightarrow (-\pi, \pi]$ provides the angle of the vector joining $(0, 0)$ to $(x, y)$ with respect to the $x-$axis. Hence, two $LG$ paths have been constructed.


Using the derived expressions for both $LG$ paths, it is desired to show the non-optimality of the second $LG$ path.

\begin{lemma} \label{lemma: non-optimality_second_LG_path}
    The second $LG$ path's length is greater than or equal to the first $LG$ path's length. Moreover, the $G$ segment's arc angle is in $[0, \pi]$ for the first $LG$ path.
\end{lemma}
\begin{proof}
    Noting that the length of an $LG$ path is given by $l_{LG} = r \phi_1 + \phi_2,$ it is desired to show that
    \begin{align} \label{eq: length_difference_LG_path}
        \Delta l_{LG} = l_{LG}^{s1} - l_{LG}^{s2} = r \left(\phi_1^{s1} - \phi_1^{s2} \right) + \left(\phi_2^{s1} - \phi_2^{s2} \right),
    \end{align}
    is non-positive for all final locations for which an $LG$ path exists. Here, $\phi_2^{si}$ is the solution for $\phi_2$ corresponding to $\phi_1^{si}$ for $i \in \{1, 2 \}.$
    To prove this claim, the final locations for which the $LG$ path(s) exists will first be parameterized to lie on a plane parallel to the $L$ segment using Lemma~\ref{lemma: regions_non_existence_LG_path}.
    A depiction of two final locations lying on one such plane is shown in Fig.~\ref{fig: LG_path_changing_diff_location}. 

    \begin{figure}[htb!]
        \centering
        \subfloat[Final location above equator or on equator along $T_0$]{\includegraphics[width = 0.4\linewidth]{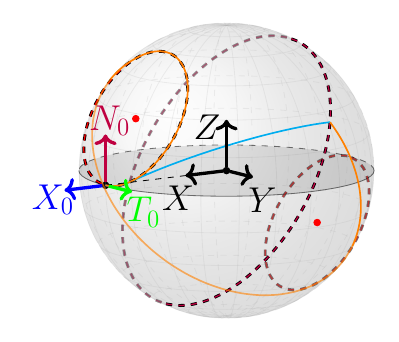}
        \label{subfig: LG_path_nonoptimality_Case_1}} \hfill
        \subfloat[Final location below equator or on equator behind $T_0$]{\includegraphics[width = 0.4\linewidth]{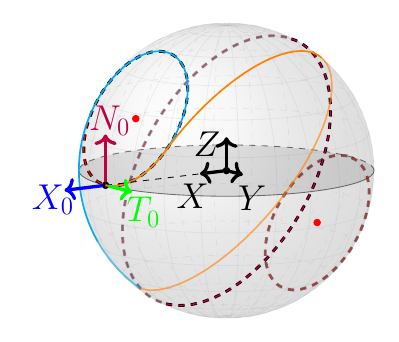}
        \label{subfig: LG_path_nonoptimality_Case_2}}
        \caption{Plane parallel to the $L$ segment, and the first (in cyan) and second (in orange) $LG$ paths for two final locations}
        \label{fig: LG_path_changing_diff_location}
    \end{figure}

    From Lemma~\ref{lemma: regions_non_existence_LG_path}, it follows that $LG$ paths exist for
    every final location lying on one of the family of planes given by 
    \begin{align} \label{eq: plane_equation_final_locations}
        \sqrt{1 - r^2} \beta_1 + r \beta_3 = d, \,\,\, d \in \left[-\sqrt{1 - r^2}, \sqrt{1 - r^2} \right].
    \end{align}
    It is desired to parameterize the final location lying on one such plane. Noting that the final location lies on the unit sphere, i.e., $\sum_{i = 1}^{3} \beta_i^2 = 1,$ and using the derived plane equation, it follows that
    $\left(\beta_1 - d \sqrt{1 - r^2} \right)^2 + r^2 \beta_2^2 = r^2 \left(1 - d^2 \right)$.
    Let $\beta_2 = \sqrt{1 - d^2} \cos{(\theta)}, \beta_1 = d \sqrt{1 - r^2} + r \sqrt{1 - d^2} \sin{(\theta)}.$ Hence, using \eqref{eq: plane_equation_final_locations} and the expression for $\beta_1$, it follows that $\beta_3 = d r - \sqrt{1 - r^2} \sqrt{1 - d^2} \sin{(\theta)}.$ 

    Using the considered parameterization, it is desired to obtain the expression for $\phi_1^{si}$ and $\phi_2^{si}$ for $i = 1, 2.$ To obtain the expression for $\phi_1^{si}$ given in \eqref{eq: solution_1_phi1_LG_path} and \eqref{eq: solution_2_phi1_LG_path}, the expression for $\gamma$ and $\delta$ must be obtained first. 
    For the considered parameterization of the final location, using the expressions for $s \gamma$ and $c \gamma$ given in \eqref{eq: definition_gamma_LG_path}, the expression for $\gamma := \atantwo \left(s \gamma, c \gamma \right)$ reduces to
	\begin{align*}
        \gamma &= \atantwo \left(\cos{(\theta)}, -\sin{(\theta)} \right) = \begin{cases}
            \frac{\pi}{2} + \theta, & 0 \leq \theta \leq \frac{\pi}{2}, \\
            -\frac{3\pi}{2} + \theta, & \frac{\pi}{2} < \theta < 2 \pi.
        \end{cases}
    \end{align*}
    The expression for $\delta,$ which is given in \eqref{eq: expression_solution_phi_1_LG_path}, for the considered parameterization is obtained after simplification as $\delta = -\frac{dr}{\sqrt{1 - r^2} \sqrt{1 - d^2}}.$ Hence, the expression for the two solutions for $\phi_1$ in \eqref{eq: solution_1_phi1_LG_path} and \eqref{eq: solution_2_phi1_LG_path} can be obtained as
    \begin{align}
        \phi_1^{s1} &= \bmod{\left(\psi (d) - \frac{\pi}{2} - \theta, 2 \pi \right)}, \label{eq: phi1s1_LG_path_simplified} \\
        \phi_1^{s2} &= \bmod{\left(\frac{3 \pi}{2} - \psi (d) - \theta, 2 \pi \right)}, \label{eq: phi1s2_LG_path_simplified}
    \end{align}
    where
    \begin{align} \label{eq: definition_psi}
        \psi (d) := \cos^{-1}{\left(\frac{-dr}{\sqrt{1 - r^2} \sqrt{1 - d^2}} \right)}.    
    \end{align}

    Using the derived parametric form for $\beta_1, \beta_2, \beta_3,$ and $\phi_1^{s1}$ and $\phi_1^{s2},$ $\cos{(\phi_2)}$ and $\sin{(\phi_2)},$ given in \eqref{eq: expression_cos_phi2_LG_path} and \eqref{eq: expression_sin_phi2_LG_path}, respectively, for the two solutions can be obtained as
    \begin{align}
        \sin{\left(\phi_2^{s1} \right)} &= \frac{\sqrt{1 - r^2 - d^2}}{\sqrt{1 - r^2}}, \,\, \cos{\left(\phi_2^{s1} \right)} = \frac{d}{\sqrt{1 - r^2}}, \label{eq: expression_sin_cos_phi2s1_parametric} \\
        \sin{(\phi_2^{s2})} &= \frac{-\sqrt{1 - r^2 - d^2}}{\sqrt{1 - r^2}}, \,\, \cos{(\phi_2^{s2})} = \frac{d}{\sqrt{1 - r^2}}. \label{eq: expression_sin_cos_phi2s2_parametric}
    \end{align}
    Noting that $\sin{(\phi_2^{s1})} \geq 0, \sin{(\phi_2^{s2})} \leq 0,$ it follows that $\phi_2^{s1} \in [0, \pi],$ $\phi_2^{s2} \in [\pi, 2 \pi) \bigcup \{0\}$ using \eqref{eq: solution_phi2_LG_path}.   

    Using the parametric form for $\phi_1$ and $\phi_2$ corresponding to the two solutions, it is desired to show that $\Delta l_{LG} \leq 0$ in \eqref{eq: length_difference_LG_path}. Noting that $\phi_1^{s1}$ and $\phi_1^{s2}$ depend on $\theta,$ whereas $\phi_2^{s1}$ and $\phi_2^{s2}$ are independent of $\theta,$ it suffices to show that
    \begin{align} \label{eq: inequality_to_be_shown_LG_path}
        f (d) = r \max_{\theta \in [0, 2\pi)} \Delta \phi_1 (\theta, d) + \Delta \phi_2 (d) \leq 0,
    \end{align}
    for every $r \in (0, 1)$ and $d \in [-\sqrt{1 - r^2}, \sqrt{1 - r^2}],$ since $\Delta l_{LG} (\theta, d) \leq f (d).$ To this end, the expressions for $\Delta \phi_1 = \phi_1^{s1} - \phi_1^{s2}$ and $\Delta \phi_2 = \phi_2^{s1} - \phi_2^{s2}$ are desired to be obtained.
    
    The expression for $\max_{\theta \in [0, 2\pi)} \Delta \phi_1$ can be obtained by deriving the explicit form for $\phi_1^{s1}$ and $\phi_1^{s2}$ given in \eqref{eq: phi1s1_LG_path_simplified} and \eqref{eq: phi1s2_LG_path_simplified}, respectively.
    Noting that $\theta \in [0, 2 \pi),$ each of the two functions can be observed to undergo a discontinuity due to the modulus function at most once. Additionally, since $\psi (d) \in \left[0, \frac{\pi}{2} \right)$ for $d \in [-\sqrt{1 - r^2}, 0)$ and $\psi (d) \in \left[\frac{\pi}{2}, \pi \right]$ for $d \in [0, \sqrt{1 - r^2}]$ from \eqref{eq: definition_psi}, the solution for $\phi_1^{s1}$ in \eqref{eq: phi1s1_LG_path_simplified} can be explicitly written for $d \in [-\sqrt{1 - r^2}, 0)$ as
    \begin{align}
        \phi_1^{s1} &= \begin{cases}
                \frac{3 \pi}{2} + \psi (d) - \theta, & \theta \in \left[0, \frac{3 \pi}{2} + \psi (d) \right], \\
                \frac{7 \pi}{2} + \psi (d) - \theta, & \theta \in \left(\frac{3 \pi}{2} + \psi (d), 2 \pi \right),
            \end{cases} \label{eq: explicit_form_phi1s1_d_neg}
    \end{align}
    and for $d \in [0, \sqrt{1 - r^2}]$ as
    \begin{align}
        \phi_1^{s1} &= \begin{cases}
            \psi (d) - \frac{\pi}{2} - \theta, & \theta \in \left[0, \psi (d) - \frac{\pi}{2} \right], \\
            \frac{3\pi}{2} + \psi (d) - \theta, & \theta \in \left(\psi (d) - \frac{\pi}{2}, 2 \pi \right).
        \end{cases} \label{eq: explicit_form_phi1s1_d_pos}
    \end{align}
    The solution for $\phi_1^{s2}$ can be explicitly written as
    \begin{align}
        \phi_1^{s2} = \begin{cases}
            \frac{3 \pi}{2} - \psi (d) - \theta, & \theta \in \left[0, \frac{3 \pi}{2} - \psi (d) \right], \\
            \frac{7 \pi}{2} - \psi (d) - \theta, & \theta \in \left(\frac{3 \pi}{2} - \psi (d), 2 \pi \right).
        \end{cases} \label{eq: explicit_form_phi1s2}
    \end{align}
    A depiction of $\phi_1^{s1}$ and $\phi_1^{s2}$ is provided in Fig.~\ref{fig: LG_path_non_optimality_first_angle_cases_with_d}.
    \begin{figure}[htb!]
        \centering
        \subfloat[Angles $\phi_1^{s1}$ and $\phi_1^{s2}$ depicted for $r = 0.3, d = -\sqrt{1 - r^2} + 0.1$]{\includegraphics[width = 0.45\linewidth]{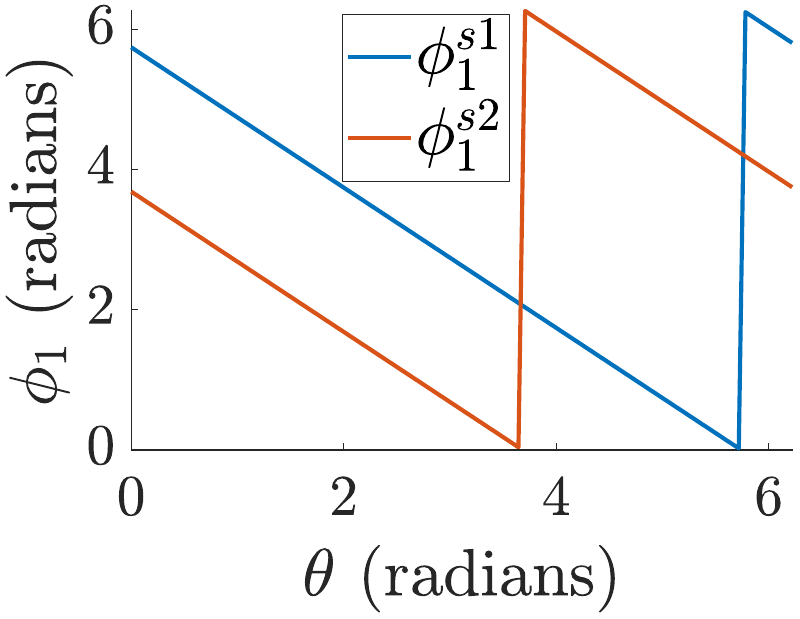}
        \label{subfig: case_1_phi1s1_angles}} \hfill
        \subfloat[Angles $\phi_1^{s1}$ and $\phi_1^{s2}$ depicted for $r = 0.3, d = \sqrt{1 - r^2} - 0.1$]{\includegraphics[width = 0.45\linewidth]{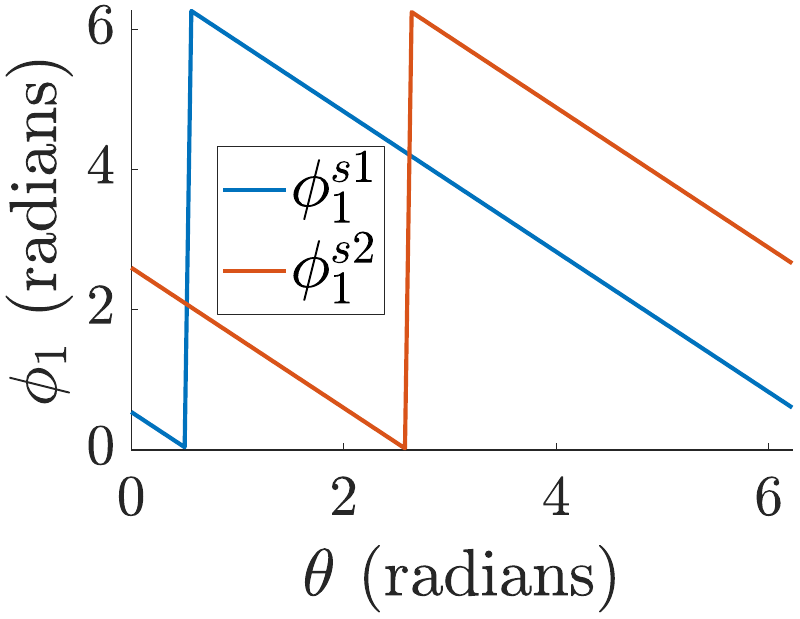}
        \label{subfig: case_2_phi1s1_angles}}
        \caption{Angles $\phi_1^{s1}$ and $\phi_1^{s2}$ for $d \in [-\sqrt{1 - r^2}, \sqrt{1 - r^2}]$}
        \label{fig: LG_path_non_optimality_first_angle_cases_with_d}
    \end{figure}

    Now, it is desired to obtain the maximum difference for $\Delta \phi_1.$ The value for $\Delta \phi_1$ for different sub-intervals of $\theta \in [0, 2 \pi)$ for $d \in [-\sqrt{1 - r^2}, 0)$ and $d \in [0, \sqrt{1 - r^2}]$ can be obtained as shown in Figs.~\ref{subfig: phi1_diff_d_less_than_zero} and \ref{subfig: phi1_diff_d_greater_than_equal_zero}, respectively. Hence, the maximum difference for $\Delta \phi_1$ can be obtained as
    \begin{align} \label{eq: max_difference_phi1s1_phi1s2}
        \max_{\theta \in [0, 2\pi)} \Delta \phi_1 = \begin{cases}
            2 \psi (d), & d \in \left[-\sqrt{1 - r^2}, \sqrt{1 - r^2} \right), \\
            0, & d = \sqrt{1 - r^2}.
        \end{cases}
    \end{align}
    It should be noted here that $d = \sqrt{1 - r^2}$ is a special case since $\psi (d) = \pi.$ Hence, $\phi_1^{s1} = \phi_1^{s2}$ from \eqref{eq: explicit_form_phi1s1_d_pos} and \eqref{eq: explicit_form_phi1s2}.
    \begin{figure}[htb!]
        \centering
        \subfloat[$\Delta \phi_1$ for $d \in [-\sqrt{1 - r^2}, 0)$]{\includegraphics[width = 0.475\linewidth]{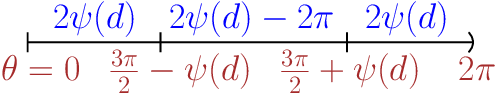}
        \label{subfig: phi1_diff_d_less_than_zero}} \hfill
        \subfloat[{$\Delta \phi_1$ for $d \in [0, \sqrt{1 - r^2}]$}]{\includegraphics[width = 0.475\linewidth]{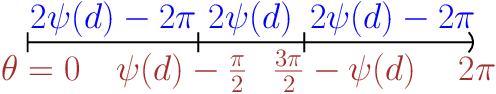}
        \label{subfig: phi1_diff_d_greater_than_equal_zero}}
        \caption{Depiction of $\Delta \phi_1$ for varying $\theta$ and $d$}
        \label{fig: difference_phi1s1_phi1s2_LG_path}
    \end{figure}

    Now, it is desired to obtain the expression for $\Delta \phi_2.$ Using the expressions for $\sin{(\phi_2^{s1})}$ and $\cos{(\phi_2^{s1})}$ given in \eqref{eq: expression_sin_cos_phi2s1_parametric}, and $\sin{(\phi_2^{s2})}$ and $\cos{(\phi_2^{s2})}$ given in \eqref{eq: expression_sin_cos_phi2s2_parametric}, $\phi_2^{s1}$ and $\phi_2^{s2}$ are obtained as shown in Fig.~\ref{fig: LG_path_non_optimality_second_angle}. From this figure, it follows that
    \begin{align} \label{eq: difference_phi2_solutions_parametric_form}
        \Delta \phi_2 = \begin{cases}
            -2 \eta (d), & d \in \left[-\sqrt{1 - r^2}, 0 \right), \\
            -\pi, & d = 0, \\
            -2 \pi + 2 \eta (d), & d \in \left(0, \sqrt{1 - r^2} \right), \\
            0, & d = \sqrt{1 - r^2}, \\
        \end{cases}
    \end{align}
    where 
    \begin{align} \label{eq: definition_eta}
        \eta (d) := \tan^{-1} \left(\frac{\sqrt{1 - r^2 - d^2}}{|d|} \right).
    \end{align}
    \begin{figure}[htb!]
        \centering
        \subfloat[$d \in [-\sqrt{1 - r^2}, 0)$]{\includegraphics[width = 0.4\linewidth]{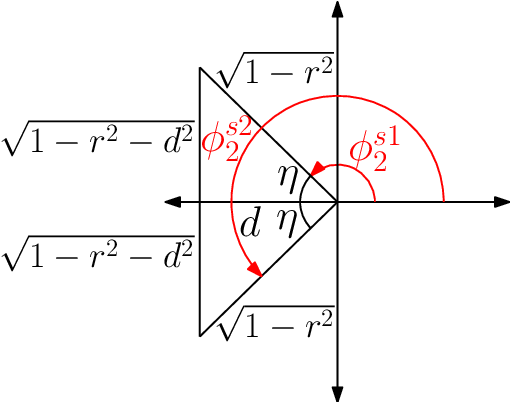}
        \label{subfig: case_1_LG_non-optimality_angle}} \hfill
        \subfloat[{$d \in$ $[0, \sqrt{1 - r^2}]$}]{\includegraphics[width = 0.4\linewidth]{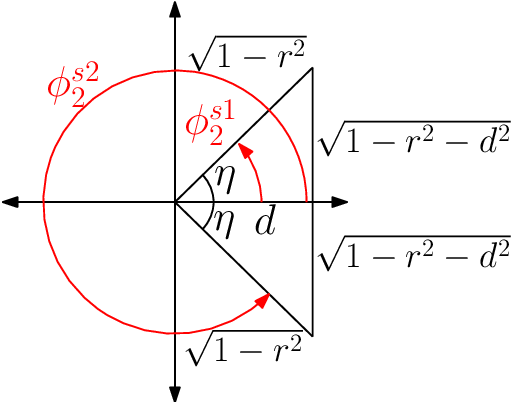}
        \label{subfig: case_3_LG_non-optimality_angle}}
        \caption{Angles $\phi_2^{s1}$ and $\phi_2^{s2}$ depicted for $d \in [-\sqrt{1 - r^2}, \sqrt{1 - r^2}]$}
        \label{fig: LG_path_non_optimality_second_angle}
    \end{figure}

    Substituting the derived expressions for $\max_{\theta \in [0, 2 \pi)} \Delta \phi_1$ and $\Delta \phi_2$
    in the left-hand side of \eqref{eq: inequality_to_be_shown_LG_path}, $f (d)$ is obtained as
    \begin{align*} 
    \begin{split}
        &f (d) = \begin{cases}
            2 r \psi (d) - 2 \eta (d), & d \in \left[-\sqrt{1 - r^2}, 0 \right), \\
            2 r \psi (d) - \pi, & d = 0, \\
            2 r \psi (d) + 2 \eta (d) - 2 \pi, & d \in \left(0, \sqrt{1 - r^2} \right), \\
            0, & d = \sqrt{1 - r^2}.
        \end{cases}
    \end{split}
    \end{align*}
    From the above equation, and using the definition of $\psi (d)$ in \eqref{eq: definition_psi} and $\eta (d)$ in \eqref{eq: definition_eta}, it can be observed that
    \begin{enumerate}
        \item $f (-\sqrt{1 - r^2}) = 0,$ $f (0) = \left(r - 1 \right) \pi,$ which is less than zero, and $f (\sqrt{1 - r^2}) = 0.$
        \item $f (d)$ is continuously differentiable for $d \in \left(-\sqrt{1 - r^2}, \sqrt{1 - r^2} \right) - \{0 \}.$
    \end{enumerate}
    Hence, it suffices to show that $f' (d) < 0$ for $d \in \left(-\sqrt{1 - r^2}, \sqrt{1 - r^2} \right) - \{0 \}$ to prove \eqref{eq: inequality_to_be_shown_LG_path}.
    The expression for $f' (d)$ can be derived to be
    \begin{align*}
        f' (d) &= \frac{-2 \sqrt{1 - r^2 - d^2}}{1 - d^2},
    \end{align*}
    in the considered intervals of $d.$
    Clearly, $f' (d) < 0,$ and hence, $f (d) \leq 0.$ As $\Delta l_{LG} (\theta, d) \leq f (d)$, the lemma proved.
    
    \textbf{Remark:} From Lemma~\ref{lemma: non-optimality_second_LG_path}, $f (d) = 0 \iff$ $d = \pm \sqrt{1 - r^2}.$ For such $d$ values, the final location lies on the $L$ segment or the diametrically opposite circle (refer to Fig.~\ref{subfig: fin_loc_LG_path_non_existence}). However, from Lemma~\ref{lemma: regions_non_existence_LG_path} and the expression for $\delta$ in \eqref{eq: expression_solution_phi_1_LG_path}, only one solution for the $LG$ path can be obtained since $\delta = \pm 1$. Hence, the two $LG$ paths are the same path. Therefore, the second $LG$ path is non-optimal for all $d \in [-\sqrt{1 - r^2}, \sqrt{1 - r^2}]$.
\end{proof}

Using Lemmas~\ref{lemma: non-optimality_C_alpha_C_pi_path}, \ref{lemma: non_optimality_Cpi_Cpi_beta}, and \ref{lemma: non-optimality_second_LG_path}, Theorem~\ref{theorem: main_theorem_free_terminal} is proved.

\section{Proof for Lemma~\ref{lemma: non-optimality_C_alpha_C_pi_path}} \label{subsect: abnormal_controls_CCpi_non_optimality}

In this section, the proof for Lemma~\ref{lemma: non-optimality_C_alpha_C_pi_path}, which pertains to non-optimality of a $C_\alpha C_\pi$ path for $r \leq \frac{\sqrt{3}}{2},$ will be provided.
\begin{proof}
    Consider an $L_\alpha R_\pi$ path, where $\alpha$ is sufficiently small. An alternate $RG$ path can be constructed, as shown in Fig.~\ref{subfig: cheaper_rg_path_r_0_4}, to attain the same final location. Due to spherical convexity, it follows that the $L_\alpha R_\pi$ path has a larger length than the $RG$ path. However, the alternate path exists only up to $r = \frac{\sqrt{3}}{2}$. For $r = \frac{\sqrt{3}}{2},$ the final location, given by $\mathbf{R}_L (\alpha) \mathbf{R}_R (\pi) e_1,$ reduces to $(-0.5, 0, -\frac{\sqrt{3}}{2})^T,$ which is independent of $\alpha.$\footnotemark \, In such a case, the alternate $RG$ path reduces to a degenerate $R$ segment. For a marginally larger $r$ and arbitrarily small $\alpha,$ 
    $\mathbf{X}_f$
    lies inside the $R$ segment, as shown in Fig.~\ref{subfig: rg_path_non_existence_r_greater_than_sqrt3_2}. For such a final location, an $RG$ path does not exist, which follows from Lemma~\ref{lemma: regions_non_existence_LG_path}. A similar argument can be used for an $R_\alpha L_\pi$ path.
    \footnotetext{The claim can be verified using the expressions for the two rotation matrices given in \cite{free_terminal_sphere}.}
    \begin{figure}[htb!]
        \centering
        \subfloat[$r = 0.4,$ $\alpha = 25^\circ$]{\includegraphics[width = 0.4\linewidth]{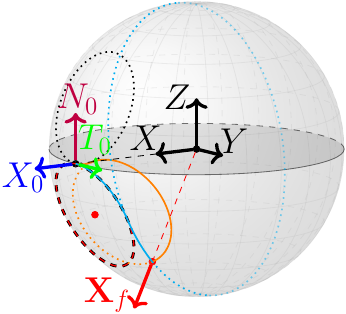}
        \label{subfig: cheaper_rg_path_r_0_4}} 
        \hfill
        \subfloat[{$r = 0.9, \alpha = 40^\circ$}]{\includegraphics[width = 0.4\linewidth]{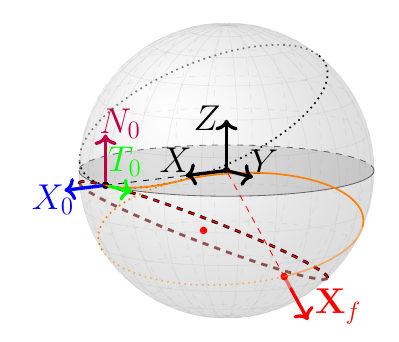}
        \label{subfig: rg_path_non_existence_r_greater_than_sqrt3_2}} 
        \caption{$RG$ path construction for $r \leq \frac{\sqrt{3}}{2}$ for non-optimality of $L_\alpha R_\pi$ path and non-existence of path for $r > \frac{\sqrt{3}}{2}$}
        \label{fig: depiction_rg_path}
    \end{figure}
\end{proof}

\section{Conclusion}

In this paper, the candidate optimal Dubins paths on a unit sphere to travel from a given initial configuration to a final location are shown to be of type $CG, CC,$ or a degenerate path of the same for $r \leq \frac{\sqrt{3}}{2}.$ Additionally, for the $CG$ paths, the arc angle of the $G$ segment is shown to be less than or equal to $\pi$ $\forall r \in (0, 1),$ which reduces the number of candidate $CG$ paths from four to two. Furthermore, for the $CC$ paths, the arc angle of the second $C$ segment is shown to be strictly greater than $\pi$ for $r \leq \frac{\sqrt{3}}{2}.$

\section*{Acknowledgment}

The authors thank Prof. Igor Zelenko from Texas A\&M University for providing the proof for Lemma~\ref{lemma: transversality_condition} \cite{communication}.

\bibliographystyle{IEEEtran}
\bibliography{IEEEabrv, cite}

\end{document}